\documentclass{IEEEtran}
\usepackage{cite}
\usepackage{amsmath,amssymb,amsfonts}
\usepackage{algorithmic}
\usepackage{graphicx}
\usepackage{textcomp}
\def\BibTeX{{\rm B\kern-.05em{\sc i\kern-.025em b}\kern-.08em
    T\kern-.1667em\lower.7ex\hbox{E}\kern-.125emX}}

\usepackage{amsthm}

\newtheorem{thm}{Theorem}
\newcommand{\dct}{\bigtriangleup_\theta}
\newcommand{\ds}{\bigtriangleup_s}
\newcommand{\curl}[1]{(\nabla\times)^{#1}}

\begin{document}
\title{A Framework for Arbitrary-Order Electromagnetic Field Approximation Based on Analytical Solutions of Maxwell’s Equations}
\author{David Wei Ge
\thanks{Dr. David Wei Ge was a software engineer, retired from Microsoft Corporation, Redmond, USA. He is an independent researcher (10218 125th AVE NE, Kirkland, WA 98033 USA. e-mail:gexiaobao.usa@gmail.com)}}

\maketitle

\begin{abstract}
The standard Yee finite-difference time-domain (FDTD) algorithm uses Maxwell's equations to replace first-order temporal derivatives with spatial curls. In this work, the relation is extended to arbitrary orders: a Time--Space Theorem establishes the correspondence between higher-order temporal derivatives and successively higher-order spatial curl operators. Based on this relation, an arbitrary-order framework for electromagnetic field approximations is developed. The framework constructs higher temporal orders through higher curl orders and permits different spatial discretizations and curl-estimation techniques, thereby providing a general basis for developing numerical algorithms of arbitrary orders. As one implementation, an arbitrary-order FDTD algorithm is constructed using staggered time and space stepping together with recursive curl estimations. Numerical simulations of an initial-value problem are performed for temporal approximation orders from second to tenth order while the spatial and temporal step sizes are kept fixed. Comparisons with the analytical solution show that the errors decrease rapidly and monotonically with increasing approximation order. These results demonstrate the practical potential of the framework for high-accuracy electromagnetic computation and provide a basis for further development of arbitrary-order numerical methods.
\end{abstract}


\section{Introduction}
\label{sec:introduction}

\IEEEPARstart{F}{inite}-difference time-domain (FDTD) methods, originating from the seminal work of Kane Yee \cite{fdtd:Yee}, remain a central tool in computational electromagnetics because of their simplicity, locality, and explicit time-stepping structure. However, the standard second-order Yee scheme is known to suffer from numerical dispersion, grid anisotropy, and accuracy limitations that become pronounced in long-time or high-frequency simulations. Although reducing the spatial and temporal step sizes can mitigate these effects, such refinement rapidly increases computational cost and, in practice, may yield diminishing improvements \cite{fdtd.cheng.2011}, \cite{GE20266756962} once the discretization becomes sufficiently fine. 

These considerations have motivated the development of higher-order FDTD schemes aimed at improving accuracy without excessive grid refinement; a systematic treatment of higher-order FDTD methods is given in \cite{fdtd.kantartzis.2006}. In many such schemes, the temporal order remains second while the spatial order is increased, commonly to fourth or sixth order   \cite{fdtd.MZhao.2019}, \cite{fdtd.Konstantinos.2004}, \cite{fdtd.bokil.2012}. Higher temporal order has also been investigated. Fourth-order temporal FDTD formulations were introduced by Fang \cite{fdtd.fang1989} and subsequently developed in later work such as Xie et al. \cite{fdtd.xie2002}. The scheme in \cite{fdtd.xie2002} employs four time levels of field history and therefore uses a (2,4) scheme during the initial three time steps before the required history becomes available.

To evaluate the benefits of higher-order schemes, numerical dispersion and the associated phase and group velocities are commonly used as accuracy measures \cite{fdtd.BLINNE2018273},\cite{fdtd.cheng.2011}, \cite{fdtd.gian.2019}. Recently developed analytical solutions of Maxwell’s equations \cite{GE2026100688}, \cite{GE20266756962} provide direct reference fields against which numerical approximations can also be evaluated.

More importantly, the analytical formulation presented in \cite{GE20266756962} can itself be used to construct a general framework for arbitrary-order electromagnetic field approximation. This framework, which is the focus of the present paper, has a structure analogous to the Taylor series:

\begin{itemize}
\item A Taylor series employs an infinite expansion in derivatives, whereas the framework employs an infinite expansion in curl operators.
\item A Taylor series provides a self-mapping of an analytic function, whereas the framework provides a mapping from initial conditions and source terms to electromagnetic fields.
\item Finite truncation of a Taylor series provide arbitrary-order approximations to an analytic function, whereas finite-order realizations of the framework provide arbitrary-order numerical approximations to the analytical solution of Maxwell’s equations.
\end{itemize}

By employing different mesh geometries, derivative orders, and estimation techniques, the framework can generate different numerical schemes. Beyond providing a unified representation of existing schemes, it also offers a systematic approach to constructing new approximation algorithms of arbitrary order.

The order of an FDTD scheme is conventionally denoted by \((N,M)\), where \(N\) and \(M\) represent the temporal and spatial orders, respectively. While the present framework accommodates this conventional interpretation, its operator-based formulation also suggests alternative meanings of temporal and spatial order, as outlined below.
\begin{itemize}

\item \textit{Temporal order.} In conventional higher-order time schemes, increased formal accuracy can be obtained by combining field approximations from multiple time levels \cite{fdtd.fang1989}, \cite{fdtd.xie2002}, \cite{fdtd:hwang}. Such temporal accuracy order is therefore associated with temporal discretization: field values at different discrete times are used to construct the time evolution. In the present framework, higher temporal order is constructed differently. Successively higher-order time derivatives are explicitly incorporated at the same time level, without requiring additional temporal discretization for these orders. The Time--Space Theorem \cite{GE20266756962} establishes a direct correspondence between time derivatives of successive orders and successively higher-order spatial curl operators. Consequently, increasing the temporal order is achieved by incorporating higher curl orders using spatial information, rather than additional sets of field values at different time levels. Thus, while both approaches characterize approximation accuracy by the order of the leading discarded term, they differ in how higher temporal accuracy is constructed: through temporal discretization in the former and through successively higher temporal derivatives, represented by higher curl operators, in the present framework. This distinction is central to the framework.
	
\item \textit{Spatial order.} Higher spatial order conventionally refers to higher-order approximations of the first-order spatial derivatives appearing in the curl operator \cite{fdtd.cheng.2011}, \cite{fdtd.review.2023}, generally through wider spatial stencils. In the present framework, spatial order has a different interpretation: it refers to the order of the curl operator itself. Wider spatial stencils contain information associated with higher-order spatial derivatives, allowing successively higher-order curl operators to be estimated. The Time–Space Theorem \cite{GE20266756962} establishes the correspondence between these higher curl orders and higher temporal-derivative orders, thereby providing a systematic means of incorporating the spatial information into time evolution.

\item \textit{Recursive approximation.} Unlike algorithms formulated for a fixed order, under the operator-based definitions of order described above, the framework has a recursive structure. Successively higher-order approximations are obtained by repeated application of the same computational procedure, without requiring a separate formulation or implementation for each order. The approximation order is therefore determined simply by the number of recursive iterations, which can be increased according to the desired accuracy.

\item \textit{Courant-number structure.} For conventional higher-order FDTD formulations, the stability constraints and associated Courant limits generally depend on the particular spatial and temporal discretizations \cite{fdtd.bokil.2012}. Within this operator-based framework, the recursive construction preserves the same Courant-number structure as the approximation order increases. Thus, successive higher-order approximations can be constructed within a consistent Courant-number formulation, providing a systematic basis for analyzing the corresponding CFL condition at arbitrary orders.

\end{itemize}

The remainder of this paper is organized as follows. Section \ref{sec:framework} develops the framework formulas. As one application of the framework, Section \ref{sec:aoFDTD} develops an arbitrary-order FDTD algorithm. Section \ref{sec:exp.ini} presents numerical results for an initial-value problem, demonstrating that the field approximations rapidly approach the analytical solution as the FDTD order increases. Finally, Section \ref{sec:conclusion} concludes the paper and discusses other potential applications of the framework.

\section{Framework Formulas}
\label{sec:framework}

Unless otherwise explicitly specified, the following notation is used throughout the paper. Capital letters denote three-dimensional vector functions, whereas lowercase letters denote scalars. Bold capital letters denote matrices, and bold lowercase letters denote vectors. Matrix elements are indexed from zero. For compactness, a matrix may be defined by listing its element indices as arguments and specifying the elements by a formula enclosed in square brackets. For example,
\begin{equation*}
	\boldsymbol{B}(i,j)_{r\times c}
	=
	[(i+1)^{j+1}]_{r\times c},
\end{equation*}
defines an \(r\times c\) matrix \(\boldsymbol{B}\), where \(i\) and \(j\) denote the row and column indices, respectively, and
\[
i=0,1,\ldots,r-1,
\qquad
j=0,1,\ldots,c-1.
\]
Thus, the element in row \(i\) and column \(j\) is defined by \((i+1)^{j+1}\).

For a vector field \(F:\mathbb{R}^3\rightarrow\mathbb{R}^3\), the repeated curl operator is defined recursively by
\begin{equation*}
	(\nabla\times)^0F=F,
\end{equation*}
and
\begin{equation*}
	(\nabla\times)^nF
	=
	\nabla\times\left((\nabla\times)^{n-1}F\right),
	\qquad n\geq1.
\end{equation*}

Consider the following three-dimensional Maxwell’s equations, transformed from their SI form using the scaled time coordinate \(\theta=ct\) \cite{maxn:Kohlmann,fdtd:meep}:
\begin{align}
	\label{3d.1}
	\frac{\partial H(x,y,z,\theta)}{\partial \theta}
	&= -\frac{1}{\eta}\nabla\times E(x,y,z,\theta),
	\\
	\label{3d.2}
	\frac{\partial E(x,y,z,\theta)}{\partial \theta}
	&= \eta\nabla\times H(x,y,z,\theta)
	-\eta J(x,y,z,\theta),
	\\
	\label{3d.3}
	H(x,y,z,0) &= F_h(x,y,z),
	\\
	\label{3d.4}
	E(x,y,z,0) &= F_e(x,y,z),
\end{align}
where
\begin{equation*}
	\eta=\sqrt{\frac{\mu}{\epsilon}},
	\qquad
	\theta=ct,
	\qquad
	c=\frac{1}{\sqrt{\epsilon\mu}},
\end{equation*}
with
\begin{equation*}
	x,y,z,t\in\mathbb{R},
	\qquad
	\epsilon,\mu\in\mathbb{R}_{>0}.
\end{equation*}
Here, \(\epsilon\) and \(\mu\) are constants, \(x,y,z\) are spatial coordinates, \(t\) is time, \(J(x,y,z,\theta)\) is the source function, and \(F_h(x,y,z)\) and \(F_e(x,y,z)\) specify the initial fields. The scaled time coordinate \(\theta=ct\), which has units of length, is used throughout the formulation; for convenience, \(\theta\) is hereafter referred to as time.

The initial fields are assumed to satisfy the following divergence constraints.
\begin{equation}
	\label{3d.5}
	\nabla\cdot F_h(x,y,z)=0
\end{equation}
\begin{equation}
	\label{3d.6}
	\nabla\cdot F_e(x,y,z)=0
\end{equation}
For notational simplicity, spatial arguments may be omitted when no ambiguity arises; for example,
\[
F(x,y,z,\theta)=F(\theta).
\]

A function pair \(\{H(x,y,z,\theta),E(x,y,z,\theta)\}\) is a solution to the problem if and only if it satisfies (\ref{3d.1})--(\ref{3d.4}). To construct approximations to the solution, we begin by separating the temporal Taylor expansions of the fields into even- and odd-order terms:
\begin{equation}
	\label{taylor.H}
	\begin{split}
		H(\theta+\dct)=H(\theta)
		&+\sum_{n=0}^{\infty}
		\frac{\partial^{2(n+1)}H(\theta)}
		{\partial\theta^{2(n+1)}}
		\frac{\dct^{2(n+1)}}{(2(n+1))!}
		\\
		&+\sum_{n=0}^{\infty}
		\frac{\partial^{2n+1}H(\theta)}
		{\partial\theta^{2n+1}}
		\frac{\dct^{2n+1}}{(2n+1)!},
	\end{split}
\end{equation}
and
\begin{equation}
	\label{taylor.E}
	\begin{split}
		E(\theta+\dct)=E(\theta)
		&+\sum_{n=0}^{\infty}
		\frac{\partial^{2(n+1)}E(\theta)}
		{\partial\theta^{2(n+1)}}
		\frac{\dct^{2(n+1)}}{(2(n+1))!}
		\\
		&+\sum_{n=0}^{\infty}
		\frac{\partial^{2n+1}E(\theta)}
		{\partial\theta^{2n+1}}
		\frac{\dct^{2n+1}}{(2n+1)!}.
	\end{split}
\end{equation}
The standard Yee-FDTD formulation uses Maxwell's equations to replace the first-order temporal derivatives by spatial curls and truncates the temporal expansion accordingly. In contrast, the present framework retains temporal derivatives of arbitrary order and uses the following Time--Space Theorem to express them systematically in terms of successively higher-order spatial curl operators.

\begin{thm}[Time--Space Theorem]
	\label{timespacetheorem}
	Let
	\[
	H(x,y,z,\theta),E(x,y,z,\theta)
	\in C^{\infty}(\mathbb{R}^4,\mathbb{R}^3)
	\]
	satisfy Maxwell's equations (\ref{3d.1}) and (\ref{3d.2}). Then, for any integer \(n\geq0\),
	\begin{align}
		\label{ts.h1}
		\frac{\partial^{2n+1}H}{\partial\theta^{2n+1}}
		&=
		\frac{1}{\eta}(-1)^{n+1}
		\curl{2n+1}E
		+J_{h,2n+1},
		\\
		\label{ts.h2}
		\frac{\partial^{2n+2}H}{\partial\theta^{2n+2}}
		&=
		(-1)^{n+1}
		\curl{2n+2}H
		+J_{h,2n+2},
		\\
		\label{ts.e1}
		\frac{\partial^{2n+1}E}{\partial\theta^{2n+1}}
		&=
		\eta(-1)^n
		\curl{2n+1}H
		+\eta J_{e,2n+1},
		\\
		\label{ts.e2}
		\frac{\partial^{2n+2}E}{\partial\theta^{2n+2}}
		&=
		(-1)^{n+1}
		\curl{2n+2}E
		+\eta J_{e,2n+2},
	\end{align}
	where the source terms are defined by
	\begin{align*}
		J_{h,1}&=0,\\
		J_{h,2n+1}
		&=
		\sum_{m=1}^{n}
		(-1)^{m+1}
		\curl{2m-1}
		\frac{\partial^{2(n-m)+1}J(\theta)}
		{\partial\theta^{2(n-m)+1}},
		 n>0,\\
		J_{h,0}&=0,\\
		J_{h,2(n+1)}
		&=
		\sum_{m=0}^{n}
		(-1)^m
		\curl{2m+1}
		\frac{\partial^{2(n-m)}J(\theta)}
		{\partial\theta^{2(n-m)}},\\
		J_{e,2n+1}
		&=
		\sum_{m=0}^{n}
		(-1)^{m+1}
		\curl{2m}
		\frac{\partial^{2(n-m)}J(\theta)}
		{\partial\theta^{2(n-m)}},\\
		J_{e,0}&=0,\\
		J_{e,2(n+1)}
		&=
		\sum_{m=0}^{n}
		(-1)^{m+1}
		\curl{2m}
		\frac{\partial^{2(n-m)+1}J(\theta)}
		{\partial\theta^{2(n-m)+1}}.
	\end{align*}
\end{thm}

A proof of the Time--Space Theorem is given in \cite{GE20266756962}. The theorem converts temporal derivatives of arbitrary order at a given time into spatial curl operators at the same time, thereby providing the mathematical basis for the operator-based temporal order introduced in Section \ref{sec:introduction}.

Using the Time--Space Theorem, the temporal derivatives in (\ref{taylor.H}) and (\ref{taylor.E}) can be expressed in terms of spatial curl operators. To construct numerical approximation algorithms, let the spatial step size be \(\ds\) and define the Courant number as
\begin{equation*}
	S_c=\frac{\dct}{\ds}.
\end{equation*}
Denote the time coordinates associated with \(H\) and \(E\) by \(\theta_h\) and \(\theta_e\), respectively. Substituting (\ref{ts.h1})--(\ref{ts.e2}) into (\ref{taylor.H}) and (\ref{taylor.E}) and expressing the resulting terms in terms of \(S_c\) yields the following framework formulas.
 
\begin{equation}
	\label{formula.H}
	\begin{split}
&H(\theta_h + \dct) = \sum_{n=0}^{N_2} (-1)^n S_c^{2n} \frac{\ds^{2n}}{(2n)!} \curl{2n}H(\theta_h)
\\
&+ \frac{1}{\eta} \sum_{n=0}^{N_1}(-1)^{n+1} S_c^{2n+1} \frac{\ds^{2n+1}}{(2n+1)!} \curl{2n+1} E(\theta_h)
\\
&+  \sum_{n=0}^{N_2} \frac{\dct^{2n}}{(2n)!} J_{h,2n}(\theta_h)
\\
&+ \sum_{n=0}^{N_1} \frac{\dct^{2n+1}}{(2n+1)!} J_{h,2n+1}(\theta_h) + O(\dct^{N})
	\end{split}
\end{equation}

\begin{equation}
	\label{formula.E}
	\begin{split}
		&E(\theta_e + \dct) = \sum_{n=0}^{N_2} (-1)^n S_c^{2n} \frac{\ds^{2n}}{(2n)!} \curl{2n}E(\theta_e)
		\\
		&+ \eta \sum_{n=0}^{N_1}(-1)^{n} S_c^{2n+1} \frac{\ds^{2n+1}}{(2n+1)!} \curl{2n+1} H(\theta_e)
		\\
		&+ \eta \sum_{n=0}^{N_2} \frac{\dct^{2n}}{(2n)!} J_{e,2n}(\theta_e)
		\\
		&+ \eta \sum_{n=0}^{N_1} \frac{\dct^{2n+1}}{(2n+1)!} J_{e,2n+1}(\theta_e) + O(\dct^{N})
	\end{split}
\end{equation}

where $N, N_1, $ and $N_2$ are integers, and
\begin{align*}
	N &\ge 2, \\
	N_1 &= 
	\begin{cases} 
		\frac{N}{2}-1, 	 & \text{if } N \text{ is even} \\
		\frac{N-1}{2}-1, & \text{if } N \text{ is odd} 
	\end{cases},
	\\
	N_2 &= 
	\begin{cases} 
		\frac{N}{2}-1, 	 & \text{if } N \text{ is even} \\
		\frac{N-1}{2}, & \text{if } N \text{ is odd} 
	\end{cases}	.
\end{align*}
Several properties of the framework formulas (\ref{formula.H}) and (\ref{formula.E}) are worth noting.
\begin{itemize}
	\item The framework formulas are algebraically independent of the spatial step size \(\ds\), since
	\[
	S_c^k\frac{\ds^k}{k!}
	=
	\frac{\dct^k}{k!}.
	\]
	The introduction of \(\ds\) and \(S_c\) therefore exposes the space--time structure of the formulas without changing the underlying temporal expansion.
	
	\item In the limit \(N\rightarrow\infty\), subject to convergence of the corresponding infinite expansions, the framework formulas recover the analytical solutions independently of the spatial step size introduced above. Representative closed-form solutions are given in \cite{GE2026100688,GE20266756962}. Such analytical solutions provide reference fields for direct accuracy evaluation of finite-order field approximations.
	
	\item Different field approximation algorithms can be constructed by employing different numerical estimations of the spatial operator terms
	
	\begin{equation}
	\label{terms.curls}
\frac{\ds^k}{k!}\curl{k}H
\qquad\text{and}\qquad
\frac{\ds^k}{k!}\curl{k}E.
	\end{equation}
	Thus, the framework is not tied to a particular spatial discretization or curl-estimation method.
	
	\item The Courant-number structure is preserved at successive approximation orders. As proved in the next section for the FDTD implementation, numerical estimations of the spatial operator terms (\ref{terms.curls}) are independent of the absolute spatial step size \(\ds\). Consequently, the effects of the spatial and temporal step sizes in the framework formulas are represented through their dimensionless ratio \(S_c=\dct/\ds\), and this structure is retained as the approximation order increases.
\end{itemize}

\section{Arbitrary-Order FDTD Algorithm}
\label{sec:aoFDTD}

As one application of the framework, we implement (\ref{formula.H}) and (\ref{formula.E}) using staggered temporal and spatial discretizations, leading to an arbitrary-order FDTD algorithm.

\subsection{Staggered Time Stepping}

Applying the framework formulas at the two half-time steps \(\theta_h\pm\dct/2\) and \(\theta_e\pm\dct/2\), respectively, and taking their differences eliminates the even-order terms, we obtain
\begin{equation}
	\label{fdtd1.H}
	\begin{split}
		&H\left(\theta_h+\frac{1}{2}\dct\right)
		-H\left(\theta_h-\frac{1}{2}\dct\right)
		\\
		&=\frac{1}{\eta}
		\sum_{n=0}^{N_1}
		\frac{(-1)^{n+1}}{2^{2n}}
		S_c^{2n+1}
		\frac{\ds^{2n+1}}{(2n+1)!}
		\curl{2n+1}E(\theta_h)
		\\
		&\quad+
		\sum_{n=0}^{N_1}
		\frac{\dct^{2n+1}}
		{2^{2n}(2n+1)!}
		J_{h,2n+1}(\theta_h)
		+O(\dct^N),
	\end{split}
\end{equation}
and
\begin{equation}
	\label{fdtd1.E}
	\begin{split}
		&E\left(\theta_e+\frac{1}{2}\dct\right)
		-E\left(\theta_e-\frac{1}{2}\dct\right)
		\\
		&=\eta
		\sum_{n=0}^{N_1}
		\frac{(-1)^n}{2^{2n}}
		S_c^{2n+1}
		\frac{\ds^{2n+1}}{(2n+1)!}
		\curl{2n+1}H(\theta_e)
		\\
		&\quad+
		\eta\sum_{n=0}^{N_1}
		\frac{\dct^{2n+1}}
		{2^{2n}(2n+1)!}
		J_{e,2n+1}(\theta_e)
		+O(\dct^N).
	\end{split}
\end{equation}
Because the staggered-time difference eliminates the even-order terms in the framework formulas, this implementation produces distinct approximations only at even temporal approximation orders. We therefore restrict \(N=2,4,6,\ldots\), for which
\[
N_1=\frac{N}{2}-1.
\]
The resulting approximations retain successively higher odd-order temporal derivatives, represented through the corresponding odd-order curl operators.

Yee-style staggered time stepping is then defined by
\begin{equation}
	\label{time.cth}
	\theta_h=q\dct,
\end{equation}
and
\begin{equation}
	\label{time.cte}
	\theta_e=q\dct-\frac{1}{2}\dct,
\end{equation}
where \(q=0,1,2,\ldots\).

For notational simplicity, square brackets are used to denote the discrete time-step index \(q\). From (\ref{time.cth}), define
\begin{equation}
	\label{time.H}
	H\left(\theta_h-\frac{1}{2}\dct\right)
	=
	H\left(q\dct-\frac{1}{2}\dct\right)
	=
	H[q].
\end{equation}
Similarly, from (\ref{time.cte}), define
\begin{equation}
	\label{time.E}
	E\left(\theta_e+\frac{1}{2}\dct\right)
	=
	E(q\dct)
	=
	E[q].
\end{equation}

Using (\ref{time.H}) and (\ref{time.E}), the staggered-time formulas (\ref{fdtd1.H}) and (\ref{fdtd1.E}) become
\begin{equation}
	\label{fdtd2.H}
	\begin{split}
		&H[q+1]=H[q]
		\\
		&+\frac{1}{\eta}
		\sum_{n=0}^{N_1}
		\frac{(-1)^{n+1}}{2^{2n}}
		S_c^{2n+1}\frac{\ds^{2n+1}}{(2n+1)!}\curl{2n+1}E[q]
		\\
		&+\sum_{n=0}^{N_1}
		\frac{\dct^{2n+1}}
		{2^{2n}(2n+1)!}
		J_{h,2n+1}(q\dct)
		\\
		&+O(\dct^N),
	\end{split}
\end{equation}
and
\begin{equation}
	\label{fdtd2.E}
	\begin{split}
		&E[q+1]=E[q]
		\\
		&+\eta
		\sum_{n=0}^{N_1}
		\frac{(-1)^n}{2^{2n}}
		S_c^{2n+1}\frac{\ds^{2n+1}}{(2n+1)!}
		\curl{2n+1}H[q+1]
		\\
		&+\eta
		\sum_{n=0}^{N_1}
		\frac{\dct^{2n+1}}
		{2^{2n}(2n+1)!}
		J_{e,2n+1}
		\left(q\dct+\frac{1}{2}\dct\right)
		\\
		&+O(\dct^N).
	\end{split}
\end{equation}

Although \(H[q]\) and \(E[q]\) share the same time-step index, they are separated by one half time step according to (\ref{time.H}) and (\ref{time.E}). Equations (\ref{fdtd2.H}) and (\ref{fdtd2.E}) therefore retain the conventional staggered-time update sequence: \(H[q+1]\) is first obtained from \(H[q]\) and \(E[q]\), followed by \(E[q+1]\) from \(E[q]\) and the newly obtained \(H[q+1]\). Unlike the standard second-order Yee update, however, each update incorporates successively higher odd-order curl operators as the temporal approximation order increases.

\subsection{Staggered Space Stepping}

Denote the field components by
\begin{equation*}
	H=[H_x,H_y,H_z]^T,
	\qquad
	E=[E_x,E_y,E_z]^T.
\end{equation*}
Yee-style spatial staggering arranges the field components as
\begin{equation}
	\label{space.H}
	H=
	\begin{bmatrix}
		H_x(x,y+\frac{1}{2}\ds,z+\frac{1}{2}\ds)\\
		H_y(x+\frac{1}{2}\ds,y,z+\frac{1}{2}\ds)\\
		H_z(x+\frac{1}{2}\ds,y+\frac{1}{2}\ds,z)
	\end{bmatrix},
\end{equation}
and
\begin{equation}
	\label{space.E}
	E=
	\begin{bmatrix}
		E_x(x+\frac{1}{2}\ds,y,z)\\
		E_y(x,y+\frac{1}{2}\ds,z)\\
		E_z(x,y,z+\frac{1}{2}\ds)
	\end{bmatrix}.
\end{equation}

The update formulas (\ref{fdtd2.H}) and (\ref{fdtd2.E}) remain compatible with the staggered grid (\ref{space.H}) and (\ref{space.E}) at arbitrary orders. On the Yee grid, application of the curl operator alternates the field locations: a curl of \(E\) is evaluated at \(H\)-locations, while a curl of \(H\) is evaluated at \(E\)-locations. Consequently, every odd-order curl of \(E\) appearing in (\ref{fdtd2.H}) is located at the corresponding \(H\)-locations, and every odd-order curl of \(H\) appearing in (\ref{fdtd2.E}) is located at the corresponding \(E\)-locations.

\subsection{Curl Estimations}
\label{sec:curls}

To estimate the spatial operator terms
\[
\frac{\ds^k}{k!}\curl{k}H
\qquad\text{and}\qquad
\frac{\ds^k}{k!}\curl{k}E
\]
in (\ref{fdtd2.H}) and (\ref{fdtd2.E}), we first establish the following result for one-dimensional derivatives.

\begin{thm}
	\label{thm.oddorderEst}
	Let \(f(s)\) be a sufficiently smooth scalar function and \(M>0\) be an integer. The scaled odd-order derivatives
	\[
	\frac{\ds^{2h+1}}{(2h+1)!}
	\frac{\partial^{2h+1}f(s)}{\partial s^{2h+1}},
	\qquad
	h=0,1,\ldots,M-1,
	\]
	can be simultaneously approximated, with an \(O(\ds^{2M})\) remainder, by constant linear combinations of \(M\) symmetric differences of nearby function values. The coefficients of these linear combinations are independent of \(\ds\).
\end{thm}

\begin{proof}
	Define the \(M\)-dimensional vector of symmetric differences
	\begin{equation*}
		\boldsymbol{d}(r)_{M \times 1} = [f(s+(r+\frac{1}{2})\ds) - f(s-(r+\frac{1}{2})\ds)]_{M \times 1}.
	\end{equation*}
	where \(r=0,1,\ldots,M-1\). Collect the scaled odd-order derivatives into the vector
	\begin{equation*}
		\boldsymbol{u}(h)_{M\times1}
		=
		\left[
		\frac{\ds^{2h+1}}{(2h+1)!}
		\frac{\partial^{2h+1}f(s)}
		{\partial s^{2h+1}}
		\right]_{M\times1},
	\end{equation*}
	where \(h=0,1,\ldots,M-1\). Expanding the symmetric differences by Taylor series gives
	\begin{equation*}
		\boldsymbol{d}
		=
		2\boldsymbol{A}\boldsymbol{u}
		+O(\ds^{2M}),
	\end{equation*}
	where
	\begin{equation*}
		\boldsymbol{A}(r,h)_{M\times M}
		=
		\left[
		\left(\frac{2r+1}{2}\right)^{2h+1}
		\right]_{M\times M}.
	\end{equation*}
	Since \(A_{rh}=a_r(a_r^2)^h\), where \(a_r=r+1/2\) are distinct and nonzero, \(\boldsymbol A\) is the product of a nonsingular diagonal matrix and a Vandermonde matrix with distinct nodes \(a_r^2\), and is therefore nonsingular.
	
	Hence,
	\begin{equation*}
		\boldsymbol{u}
		=
		\frac{1}{2}\boldsymbol{A}^{-1}\boldsymbol{d}
		+O(\ds^{2M}).
	\end{equation*}
	Since \(\boldsymbol{A}\) depends only on the integer indices \(r\) and \(h\), both \(\boldsymbol{A}\) and \(\boldsymbol{A}^{-1}\) are independent of \(\ds\). Therefore, each scaled odd-order derivative is approximated by a constant linear combination of the nearby function values, with coefficients independent of \(\ds\).
\end{proof}

\subsection{Notes on Implementations}

Theorem \ref{thm.oddorderEst} can be applied directly to estimate first-order curls. For \(N=2\), only first-order curls are required, and (\ref{fdtd2.H}) and (\ref{fdtd2.E}) reduce to the standard Yee algorithm. Higher-order implementations require additional considerations, as outlined below.
\begin{itemize}
\item \textbf{Spatial order.} According to Theorem \ref{thm.oddorderEst}, a temporal approximation order \(N\) requires a spatial estimation order \(2M\) satisfying
\[
2M\geq N.
\]
Thus, \(2M = N\) is the minimum spatial order required to fully realize the specified temporal order. Higher spatial orders, \(2M>N\), may also be employed. Determining whether and under what conditions such higher spatial orders provide additional accuracy or other numerical benefits is an interesting topic for further investigation.

\item \textbf{Recursive curl estimations.} Theorem \ref{thm.oddorderEst} provides one-dimensional derivative estimations and therefore does not directly account for mixed partial derivatives, such as \(\partial^2/(\partial x\,\partial y)\), that arise in higher-order curls. Such curl terms can instead be constructed recursively by repeated application of the theorem.

\item \textbf{Boundary-layer requirement.} Higher-order derivative estimations employ wider spatial stencils and consequently reduce the region in which the curls can be estimated entirely from interior field values. For a \(2M\)-th order estimation, Theorem \ref{thm.oddorderEst} requires \(M\) neighboring grid layers. Each recursive application therefore reduces the region of valid curl estimations by \(M\) layers. Since the construction requires \(M\) recursive applications, the valid interior region is reduced by \(M^2\) grid layers. A boundary layer of at least \(M^2\) cells is therefore required to isolate the interior region from inaccurate curl estimations near the computational boundary.	
\end{itemize}

\section{Example: An Initial-Value Problem}
\label{sec:exp.ini}

For this example, the source is set to zero and the following initial fields are used:
\begin{align}
	\label{3d.exa.ini.H}
	F_h(x,y,z)&=0,
	\\
	\label{3d.exa.ini.E}
	F_e(x,y,z)&=\exp(-r^2)[yz,-2zx,xy]^T,
\end{align}
where $r^2=x^2+y^2+z^2$.

Since the physical initial conditions are specified at \(\theta=0\), an initialization step is required for the staggered-time representation. According to (\ref{time.H}), \(H[0]\) represents \(H(-\dct/2)\); therefore, directly setting \(H[0]=H(0)=0\) would shift the prescribed initial condition by one-half time step, an initialization issue inherited from the conventional Yee time staggering. The general framework formula (\ref{formula.H}) provides a way to address this problem by approximating \(H(-\dct/2)\) from the fields prescribed at \(\theta=0\). For the present example, since \(H(0)=0\) and \(J=0\), all even-order curl and source terms vanish, and the general framework formula (\ref{formula.H}) becomes one half of the \(H\)-update in (\ref{fdtd2.H}), with a minus sign for backward time stepping. Thus, \(H[0]\) can be initialized using the same arbitrary-order curl-estimation procedure as the subsequent FDTD updates.

Five simulations are performed using temporal orders
\[
N=2,4,6,8,10,
\]
respectively, with the spatial estimation order chosen as \(2M=N\). The same spatial and temporal step sizes are used in all five simulations:
\begin{align*}
	\ds&=0.4,\\
	\dct&=0.23.
\end{align*}
Thus, differences among the simulations result from the approximation order rather than from mesh refinement.

Electromagnetic field snapshots are recorded over 26 time steps within a central \(8\)-m cube defined by
\begin{equation}
	\label{num.domain}
	\begin{split}
		x&=i\ds,\qquad y=j\ds,\qquad z=k\ds,\\
		i,j,k&=0,\pm1,\pm2,\ldots,\pm10,\\
		q&=0,1,2,\ldots,25.
	\end{split}
\end{equation}
The corresponding times of \(H[q]\) and \(E[q]\) are given by (\ref{time.H}) and (\ref{time.E}), respectively. For each simulation, the computational domain is chosen sufficiently large that boundary effects do not reach the central \(8\)-m cube during the recorded time interval.

A closed-form analytical solution for this problem is available in \cite{GE20266756962}, allowing the approximation errors to be evaluated directly at the staggered locations of the FDTD field components. At each grid index \((i,j,k)\), denote the FDTD output by
\begin{equation*}
	E_{\mathrm{fdtd}}
	=
	[E_{x,\mathrm{fdtd}},E_{y,\mathrm{fdtd}},E_{z,\mathrm{fdtd}}]^T,
\end{equation*}
where the three components are located at their respective staggered-grid positions defined by (\ref{space.E}). The analytical field components are evaluated individually at these same staggered locations and collected as
\begin{equation*}
	E_{\mathrm{analytic}}
	=
	[E_{x,\mathrm{analytic}},E_{y,\mathrm{analytic}},E_{z,\mathrm{analytic}}]^T.
\end{equation*}
Although these quantities are written in vector form for convenience, their components correspond to different spatial locations on the staggered grid.

Define the pointwise field error by
\begin{equation}
	\label{err.point}
	e_E(i,j,k,q)
	=
	\left\|
	E_{\mathrm{analytic}}-E_{\mathrm{fdtd}}
	\right\|_2,
\end{equation}
where
\begin{equation*}
	\left\|
	E_{\mathrm{analytic}}-E_{\mathrm{fdtd}}
	\right\|_2
	=
	\sqrt{
		(\Delta E_x)^2+
		(\Delta E_y)^2+
		(\Delta E_z)^2
	},
\end{equation*}
and \(\Delta E_\alpha=E_{\alpha,\mathrm{analytic}}-E_{\alpha,\mathrm{fdtd}}\), \(\alpha=x,y,z\).

The following two error statistics are then calculated over the domain defined by (\ref{num.domain}):
\begin{equation}
	Err_{\max}(q)
	=
	\max_{i,j,k} e_E(i,j,k,q),
\end{equation}
and
\begin{equation}
	Err_{\mathrm{sum}}(q)
	=
	\sum_{i,j,k}e_E(i,j,k,q).
\end{equation}

Fig.~\ref{fig:errMax} compares \(Err_{\max}(q)\) obtained from the five simulations with temporal orders \(N=2,4,6,8,\) and \(10\), while Fig.~\ref{fig:errSum} compares the corresponding \(Err_{\mathrm{sum}}(q)\). For this example, both error measures decrease rapidly toward zero as the approximation order increases, while the spatial and temporal step sizes remain unchanged.

For comparison, the mesh-refinement study of the same initial-value problem in \cite{GE20266756962} showed that, under the same configurations considered here, the simulation errors approached nonzero limiting values as the mesh was refined. The present results demonstrate that increasing the approximation order provides an additional means of reducing errors beyond those obtained through mesh refinement alone.

To make the error reduction at higher orders more visible, Fig.~\ref{fig:errMaxLog} and Fig.~\ref{fig:errSumLog} show \(\ln[Err_{\max}(q)]\) and \(\ln[Err_{\mathrm{sum}}(q)]\), respectively. At the initial time step \(q=0\), the errors are zero and their logarithms are therefore undefined, with the limiting value \(\ln(0)=-\infty\). For plotting purposes, these initial values are represented by values lower than the minimum finite error level shown in each figure. The logarithmic plots clearly show that, at each recorded time step after initialization, both error measures decrease monotonically as the approximation order increases.

\section{Conclusion}
\label{sec:conclusion}
For standard FDTD, continued mesh refinement may eventually provide diminishing improvements in accuracy \cite{fdtd.cheng.2011,GE20266756962}. In contrast, the arbitrary-order framework developed in this paper provides another route to increasing accuracy while keeping the spatial and temporal step sizes fixed. Subject to convergence, the framework approaches the analytical solution as the approximation order tends to infinity. More importantly for practical computation, the numerical results presented here show that the simulation errors decrease rapidly with increasing order and become small at moderate finite orders. Thus, approaching the analytical solution does not necessarily require an impractically high approximation order.

An important topic for future work is the tradeoff between accuracy and computational cost. The computational cost depends strongly on the particular implementation of the framework formulas. For example, by applying (\ref{formula.E}) symmetrically about \(\theta_e\), we obtain
\begin{equation*}
	\begin{split}
		&E(\theta_e+\dct)+E(\theta_e-\dct)
		\\
		&=2\sum_{n=0}^{N_2}(-1)^n S_c^{2n}
		\frac{\ds^{2n}}{(2n)!}
		\curl{2n}E(\theta_e)
		\\
		&\quad+2\eta\sum_{n=0}^{N_2}
		\frac{\dct^{2n}}{(2n)!}
		J_{e,2n}(\theta_e)
		+O(\dct^N).
	\end{split}
\end{equation*}
This implementation provides an arbitrary-order approximation of \(E\) involving only \(E\) and the source terms, without requiring \(H\). For applications in which the magnetic field is not required, such a formulation can avoid the computation and storage associated with \(H\), potentially reducing the computational cost substantially.

Although the framework supports the conventional Yee staggered grid, it also permits implementations on non-staggered grids. Such implementations may be useful for several reasons.
\begin{itemize}
\item \textbf{Collocated field vectors.} On a non-staggered grid, the components of the electromagnetic fields are evaluated at common spatial locations. Vector quantities such as the Poynting vector can therefore be constructed directly, without the interpolation or averaging required to collocate field components on a Yee grid.
\item \textbf{Boundary treatments.} Non-staggered grids can simplify the implementation of certain boundary treatments. In our implementations, PEC boundary conditions and PML layers have been applied straightforwardly on non-staggered grids with satisfactory numerical results, although a detailed study of these implementations is outside the scope of the present paper.
\item \textbf{Higher-order spatial stencils.} Staggering provides its greatest relative reduction in sampling distance for immediately neighboring field components. As increasingly wide stencils are employed, this relative geometric advantage decreases. It is therefore of interest to investigate the relative accuracy and computational characteristics of staggered and non-staggered grids for higher-order implementations.
\end{itemize}

\begin{figure}[!t]
	\centerline{\includegraphics[width=\columnwidth]{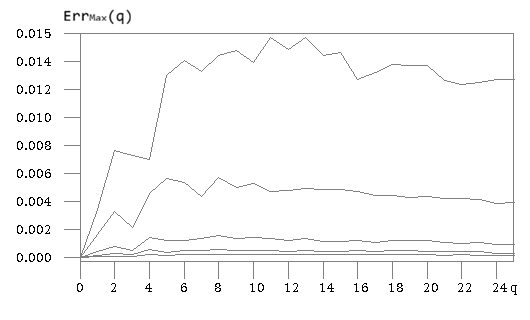}}
	\caption{Maximum error at each time step, \(Err_{\max}(q)\). The curves correspond, from top to bottom, to temporal approximation orders \(N=2,4,6,8,\) and \(10\), respectively.}
	\label{fig:errMax}
\end{figure}
\begin{figure}[!t]
	\centerline{\includegraphics[width=\columnwidth]{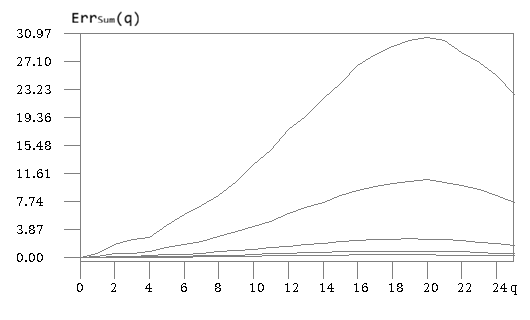}}
	\caption{Summed error at each time step, \(Err_{\mathrm{sum}}(q)\). The curves correspond, from top to bottom, to temporal approximation orders \(N=2,4,6,8,\) and \(10\), respectively.}
	\label{fig:errSum}
\end{figure}

\begin{figure}[!t]
	\centerline{\includegraphics[width=\columnwidth]{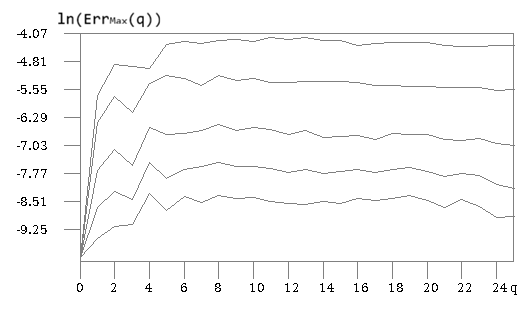}}
	\caption{Natural logarithm of the maximum error at each time step, \(\ln(Err_{\max}(q))\). The curves correspond, from top to bottom, to temporal approximation orders \(N=2,4,6,8,\) and \(10\), respectively.}
	\label{fig:errMaxLog}
\end{figure}
\begin{figure}[!t]
	\centerline{\includegraphics[width=\columnwidth]{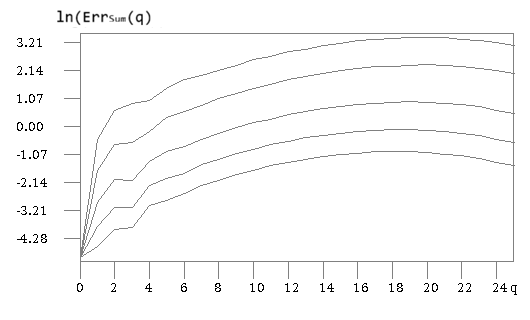}}
	\caption{Natural logarithm of the summed error at each time step, \(\ln(Err_{\mathrm{sum}}(q))\). The curves correspond, from top to bottom, to temporal approximation orders \(N=2,4,6,8,\) and \(10\), respectively.}
	\label{fig:errSumLog}
\end{figure}

\section*{Acknowledgment}

During the preparation of this manuscript, the author used ChatGPT to assist with language refinement and critical discussion of mathematical exposition. After using this tool, the author carefully reviewed and edited all content and takes full responsibility for the published manuscript.

\end{document}